\documentclass[11pt, reqno]{amsart}
\usepackage{bm}
\usepackage{multirow}
\usepackage{tikz}
\usepackage{caption}
\usepackage{textgreek}
\usetikzlibrary{
 decorations.pathmorphing,%
 decorations.markings,%
 decorations.pathreplacing,%
 decorations.text,%
 calc,%
 patterns,%
 shapes.geometric,%
 arrows,
 decorations.shapes,
 positioning,
 plotmarks,
 shadings,
 math,
 intersections
}

\definecolor{jeffColor}{RGB}{102, 0, 204}
\definecolor{yaizaColor}{RGB}{0, 153, 153}
\definecolor{certainty}{RGB}{64, 228, 198}
\definecolor{hope}{RGB}{228, 194, 64}

\definecolor{periodColor}{RGB}{255, 167, 105}
\definecolor{dark-green}{RGB}{135, 194, 130}
\tikzset{>=latex} 
\tikzset{font=\small}
\tikzset{mark size=1.5pt, mark options=thin}
\tikzset{pin distance=4pt,
 every pin edge/.style={<-, thin, shorten <= -2pt}}
\usepackage{array,booktabs,tabularx}
\usepackage{graphicx}
\usepackage{amssymb}
\usepackage{arydshln}
\usepackage{enumerate}
\usepackage{color}
\usepackage{esint}
\usepackage{mathtools}
\usepackage{dsfont}
\usepackage{ esint }
\usepackage{placeins}
\usepackage[show]{ed}
\usepackage[final]{showkeys}
\usepackage{mathrsfs}
\usepackage{changepage}   
\usepackage{comment}
\usepackage{enumitem}
\usepackage{soul}
\usepackage{needspace}
\usepackage[normalem]{ulem}

\usepackage{multicol}
\usepackage{tabto}
\usepackage{hyperref}

\definecolor{uipoppy}{RGB}{221,128,71}
\definecolor{uipaleblue2}{RGB}{179,196,215}
\definecolor{uiviolet}{RGB}{86,86,99}
\definecolor{uiblack}{RGB}{0, 0, 0}
\definecolor{azul}{RGB}{0,128,255}
\definecolor{verde}{RGB}{50,180,50}
\definecolor{pale-verde}{RGB}{155,207,145}
\definecolor{uipaleblue}{RGB}{108,199,220}
\definecolor{light-gray}{RGB}{198,198,198}

\newtheorem{lemma}{Lemma}
\newtheorem{theorem}[lemma]{Theorem}

\newtheorem{corollary}[lemma]{Corollary}

\theoremstyle{definition}
\newtheorem{definition}[lemma]{Definition}
\newtheorem{remark}[lemma]{Remark}

\usepackage{dsfont}

\newcommand{\mc}[1]{\mathcal{#1}}

\newcommand{\blue}[1]{{{#1}}}

\def\XXint#1#2#3{{\setbox0=\hbox{$#1{#2#3}{\int}$} \vcenter{\hbox{$#2#3$}}\kern-.5\wd0}}

\DeclareMathOperator{\diam}{diam}

\DeclareMathOperator{\supp}{supp}

\newcommand{\e}{\varepsilon}

\numberwithin{equation}{section}
\numberwithin{lemma}{section}

\newcommand{\someLetter}{\blue{\mc{S}}}

\title{Lower bounds for piecewise polynomial approximations of oscillatory functions}

\author{Jeffrey Galkowski}
\address{Department of Mathematics, University College London, London, UK}
\email{j.galkowski@ucl.ac.uk}

\begin{document}
\begin{abstract}
We prove lower bounds on the error incurred when approximating any oscillating function using piecewise polynomial spaces. \blue{The estimates are explicit in the polynomial degree and have optimal dependence on the meshwidth and frequency when the polynomial degree is fixed.}  These lower bounds, for example, apply when approximating solutions to Helmholtz plane wave scattering problem.
\end{abstract}
\maketitle

\section{Introduction}

In this article, we study the error incurred when approximating highly oscillatory functions using piecewise polynomial spaces. This type of space is standard when using both finite element and boundary element methods to numerically approximate solutions to partial differential equations (PDE). We are motivated by the application of these methods to solve high frequency problems. For example, to solve the Helmholtz sound-soft or sound-hard scattering problem:
\begin{equation}
\label{e:obstacle}
\begin{gathered}
(-\Delta -k^2)u=f\,\,\text{in }\mathbb{R}^d\setminus \Omega,\qquad u|_{\partial\Omega}=g,\qquad (\partial_r-ik)u=o_{r\to \infty}(r^{\frac{1-d}{2}}),\\
(-\Delta -k^2)u=f\,\,\text{in }\mathbb{R}^d\setminus \Omega,\qquad \partial_\nu u|_{\partial\Omega}=g,\qquad (\partial_r-ik)u=o_{r\to \infty}(r^{\frac{1-d}{2}}),
\end{gathered}
\end{equation}
where $\Omega\Subset \mathbb{R}^d$ is a bounded domain with connected complement
or the variable wave speed Helmholtz problem:
\begin{equation}
\label{e:inHome}
-\partial_{x_j}(a^{ij}\partial_{x_i}u)-k^2c^{-2}(x)u=f\,\,\text{in }\mathbb{R}^d,\qquad (\partial_r-ik)u=o_{r\to \infty}(r^{\frac{1-d}{2}}),
\end{equation}
where $a^{ij}(x)\equiv \delta^{ij}$ and $c(x)\equiv 1$ for $|x|\gg 1$. In both cases the solution, $u$, with data coming from a scattering problem will oscillate at frequency $k$ in a sense to be made precise below. Since numerical methods such as the Galerkin method seek to approximate $u$ in some finite dimensional space, $\mc{V}_k$, it is important to understand what the best possible approximation error for such oscillating functions is. Indeed, a numerical method for a high frequency PDE is frequency-robust quasi-optimal if the error in the method is controlled uniformly by the best approximation error in the relevant finite dimensional space; i.e. if the numerical solution, $u_{num}$, satisfies
$$
\|u-u_{num}\|\leq C_{\rm{qo}}\inf_{v\in \mc{V}_k}\|u-v\|
$$
where $u$ is the exact solution and $C_{\rm{qo}}>0$ is a constant that is uniform over $k>1$. There has been a great deal of effort in understanding when numerical methods based on piecewise polynomial spaces are frequency-robust quasi-optimal (see e.g.~\cite{LaSpWu:22,GaLaSpWu:21,MeSa:11,GaSp:22,Ih:98,IhBa:97,IhBa:95,MeSa:10, ChNi:20} and references there-in).

Upper bounds on the error for piecewise polynomial approximations are completely standard in the literature (see e.g.~\cite[Section 3.1]{Ci:02}~\cite[Chapter 4]{Bre:08}~~\cite[Chapter 4]{SaSc:11}).  In this article, we prove complementary, optimal lower bounds on the error when approximating \emph{any} oscillatory function by piecewise polynomials and hence, on the absolute error for many frequency-robust quasi-optimal methods (see Section~\ref{s:approx-k} for more detail). Furthermore, in forthcoming work~\cite{GaRaSp:22}, we will use these estimates to show that the standard second-kind boundary element methods for trapping Dirichlet problems and, even non-trapping, Neumann problems suffer from the pollution effect i.e. that, \blue{when the polynomial degree is fixed,} the mesh-width, $h$, must satisfy $hk=o(1)$ in order to maintain accuracy as the frequency increases.

We now state a consequence of the main theorem of this paper (Theorem~\ref{t:ML2}) informally.
\begin{theorem}
\label{t:RdL2Informal}
Let $0<\Xi_L<\Xi_H$. Then there are $k_0>0$ and $c>0$ such that for all $p\in 0,1,\dots$, $k>k_0^{p+1}$, all $u\in L^2(\mathbb{R}^d)$ oscillating with frequency between $\Xi_Lk$ and $\Xi_Hk$,  all $0<h<1$, and all piecewise polynomials, $v_h$, of degree $p$ on a regular mesh with scale $h$
\begin{equation}
\label{e:lowL2Rd1}
\Big(\frac{chk}{p^2}\Big)^{p+1}\|u\|_{L^2(\mathbb{R}^d)}\leq \|u-v_h\|_{L^2(\mathbb{R}^d)}.
\end{equation}
\end{theorem}
\begin{remark}The precise definition of a piecewise polynomial on a regular mesh is given in Section~\ref{s:mesh} and of the concept of oscillating with a certain frequency in Section~\ref{s:oscDef}.\end{remark}

\begin{remark}Note that we have used a mesh on all of $\mathbb{R}^d$ to simplify the statement of the theorem. Such a mesh can easily be defined, for example, by constructing a mesh on $[0,1]^d$ and tiling $\mathbb{R}^d$ with this mesh.\end{remark}

A standard assumption in numerical analysis of high-frequency Helmholtz problems is that the meshwidth $h$ must satisfy $hk \lesssim 1$ (see e.g.~\cite{IhBa:95}). Here, we use the notation $a\lesssim b$ if $a\leq Cb$ for some constant independent of $k$. This can be justified heuristically using the Shannon-Nyquist-Whittaker sampling theorem. However, this theorem holds only in one dimension and for functions with compactly-supported Fourier transform.  Theorem~\ref{t:RdL2Informal}, or more precisely, Theorem~\ref{t:ML2}, shows that, \blue{at least when $p$ is bounded independent of $k$,} $hk\ll1$  is required to approximate a function oscillating with frequency $\sim k$ (i.e. $ck\leq \text{frequency}\leq Ck$) accurately using piecewise polynomials. In other words, for piecewise polynomials $\sim k^d$ degrees of freedom are required to approximate such a function and hence, no method based on polynomials \blue{of fixed degree} in dimension $d$ can be accurate without at least this many degrees of freedom.

Despite the fact that they have many natural applications in numerical analysis, lower estimates on the approximation error for oscillatory functions are absent in the literature. Indeed, the only lower estimates on approximation by finite dimensional spaces of which the author is aware concern the Kolmogorv $n$-width (see~\cite{Je:72} and references there-in). These estimates assert the existence e.g. of some $H^{p+1}$ function, $u$, not necessarily oscillating at any particular frequency which achieves~\eqref{e:lowL2Rd1}. This existence result says nothing about the structure of $u$ nor how many such $u$ there are (see Section~\ref{s:nWidth} for a more detailed discussion).  Because of this, it is not useful for giving lower estimates on the approximation error in practice for many numerical problems.

\begin{remark}
\blue{The reader may be aware of the lower bound on polynomial approximation provided by Bramble--Hilbert~\cite{BrHi:70}
$$
|u|_{H^{p+1}(\Omega)}\leq \inf_{v\in\mathbb{P}_{p}}\|u-v\|_{H^{p+1}(\Omega)},
$$
where $\mathbb{P}_{p}$ denotes the space of polynomials of degree $p$ and $|u|_{H^{p+1}(\Omega)}$ denotes the $(p+1)^{\text{st}}$ Sobolev seminorm. However, this estimate is not useful in practice since it involves an equal number of derivatives on the left-hand and right-hand sides. Moreover, this estimate is essentially trivial since $|v|_{H^{p+1}}=0$ and $|u|_{H^{p+1}}\leq \|u\|_{H^{p+1}}$. }

\blue{
In fact, one requires \emph{both} control from below on $(p+1)^{th}$ order derivatives \emph{and} from above on high derivatives of $u$ to obtain an estimate like~\eqref{e:lowL2Rd1}. For instance, consider $u(x):= 1+\e^{p+1}\sin(\e^{-1}k x)$ on $[0,1]$, with the standard mesh by intervals $(j/N,(j+1)/N)$ $j=0,\dots, N-1$. Then, with $\mathbb{P}_{p,N}$ the piecewise polynomials of degree $p$, on this mesh, we have
$$
\inf_{v\in\mathbb{P}_{0,N}}\|u-v\|_{L^2(0,1)}\leq C\e^{p+1}\ll (k/N)^{p+1}\|u\|_{L^2(0,1)},\qquad \inf_{\e>0}|u|_{H^{p+1}(0,1)}\gtrsim k^{p+1}
$$
provided that $\e\ll k/N$. }
\end{remark}

Proving Theorem~\ref{t:RdL2Informal} involves two substantial new difficulties relative to existing results. First, for a given $u$, unlike for the corresponding upper bounds, it is not possible to prove~\eqref{e:lowL2Rd1} by construction of an interpolating polynomial. One must instead consider \emph{all} possible piecewise polynomial and \emph{all} possible regular meshes simultaneously and hence the proof must be based on the inherent structure of piecewise polynomial spaces. Second, since we want the estimate~\eqref{e:lowL2Rd1} for \emph{all} possible oscillating functions, it is not sufficient to construct a single bad oscillating function and again one must use instead the structure inherent in the space of oscillating functions.

\subsection{Definitions of meshes and polynomial spaces}
\label{s:mesh}
We work with piecewise polynomial finite element spaces. In order to describe these spaces, we first introduce regular meshes of an (open) Riemannian manifold  $(M,g)$ of dimension $d$, possibly with Lipschitz boundary. In practice, $M$ is usually either a domain $\mathscr{O}$ in $\mathbb{R}^n$ (with $d=n$)  and the standard Euclidean metric or a hypersurface $\Gamma$ embedded in $\mathbb{R}^n$ (with $d=n-1$) endowed with the metric inherited from $\mathbb{R}^n$. In order to place these two cases in a uniform framework, we use the language of Riemannian manifolds.
\begin{definition}[meshes for $M$]
Let $\someLetter\Subset \mathbb{R}^d$ be open with Lipschitz boundary. A  \emph{mesh for $M$ with reference element $\someLetter$, $\mc{T}$,} is a locally finite collection of open subsets of $M$ such that the following holds:
\begin{enumerate}
\item The open sets are disjoint in the sense that if $T_1,T_2\in \mc{T}$ and $T_1\cap T_2\neq \emptyset$, then $T_1=T_2$. 
\item $\mc{T}$ covers $M$ in the sense that $\overline{M}=\bigcup_{T\in \mc{T}}\overline{T}.$
\item For every $T\in \mc{T}$, there is $p\geq 1$ and a bijection $\gamma_{T}:\overline{\someLetter}\to \overline{T}$ such that
$$
\sup_{T\in \mc{T}}\sup_{x\in \overline{T}}\sup_{|\alpha|\leq p}\|\partial_x^\alpha\gamma_T(x)\|+\|(D\gamma_T)^{-1}(x)\|<\infty,
$$
where $D\gamma_T$ denotes the Jacobian of $\gamma_T$.
\end{enumerate}
We say that $\mc{T}$ is a \emph{mesh for $M$} if there is $\someLetter\Subset \mathbb{R}^d$ such that $\mc{T}$ is a mesh for $M$ with reference element $\someLetter$. For $R>0$ and $p\in \{1,\dots\}$ we say that $\mc{T}$ is \emph{$(p,R)$ regular} if there are $\{\gamma_T\}_{T\in \mc{T}}$ such that
\begin{equation}
\label{e:coordinates}
\sup_{T\in \mc{T}}\sup_{|\alpha|\leq p}\sup_{x\in \overline{T}}\blue{\frac{1}{|\alpha|}\Big(\big(\|\partial_x^\alpha \gamma_T(x)\|\big)^{\frac{1}{|\alpha|}}}+\|(D\gamma_T)^{-1}(x)\|\Big)\leq\blue{ 1+R|\alpha|^{-1}}. 
\end{equation}
We call a collection $\{\gamma_T\}_{T\in \mc{T}}$ satisfying~\eqref{e:coordinates} a \emph{$(p,R)$-regular set of coordinates for $\mc{T}$}.
\end{definition}
\begin{remark}
It is easy to see that for any $p\geq 1$, a $(p,R)$ regular mesh is shape regular with shape regularity constant $\lesssim R^2$. Indeed, 
$$
\diam\big(\gamma_T(\someLetter)\big)\leq C_{\someLetter}\sup_{x\in\overline{T}}\sup_{|\alpha|=1}\|\partial_x^\alpha \gamma_T(x)\|\leq C_\someLetter R,
$$
and, if $B(x_0,r_\someLetter)\subset \someLetter$, then $B(\gamma_T(x_0),R^{-1}r_{\someLetter})\subset \gamma_T(\someLetter)$. 
 
\end{remark}
\begin{definition}[Broken Sobolev spaces]
For a mesh $\mc{T}$, we define the broken Space, $H_{\mc{T}}^\ell(M)\subset L^2(M)$ to be the set of $u\in L^2(M)$ such that $u|_{\mc{T}}\in H^\ell(T)$ for all $T\in \mc{T}$ and endow it with the norm
$$
\|u\|_{H_{\mc{T}}^\ell(M)}^2:=\|u\|_{L^2(M)}^2+\sum_{T\in\mc{T}}\|u|_{T}\|_{H^\ell(T)}^2.
$$
\end{definition}
\begin{remark}
Observe that if $u\in H^\ell(M)$, then $\|u\|_{H_{\mc{T}}^\ell(M)}=\|u\|_{H^\ell(M)}$ for any mesh $\mc{T}$. 
\end{remark}

We introduce the notion of $(p,R)$ regularity in~\eqref{e:coordinates} because we are interested in uniform estimates as the size of a mesh element decreases. In order to do this, we need to assume that as the mesh elements decrease in size, their behavior does not become too wild. This will be encoded using $(p,R)$ regularity.

Below, we will actually work with families of meshes at decreasing scale. To do this, we make the following definition.
\begin{definition}[Scales of meshes for $M$]
\label{d:scale}
\blue{Let $ p\in\{1,\dots\}\cup\{\infty\}$ and $\someLetter\Subset \mathbb{R}^d$ open with Lipschitz boundary}. A \emph{$p$-scale of meshes for $M$ with reference element $\someLetter$} is a set $I\subset (0,1)$ with $ 0\in \overline{I}$ and a collection of meshes for $M$,  $\{\mc{T}_{h}\}_{h\in I}$, such that $\mc{T}_{h}$ is a mesh for $M$ with reference element $h\someLetter$ and there is $R>0$ such that for all $h\in I$, $\mc{T}_{h}$ is $(p,R)$ regular. 

We say that $\mc{M}:=(I,\{\mc{T}_{h}\}_{h\in I})$ is a \emph{$p$-scale of meshes for $M$} if there is $\someLetter$ as above such that $\mc{M}$ is a $p$-scale of meshes for $M$ with reference element $\someLetter$.  We say that $\mc{M}$ is a \emph{scale of meshes for $M$} if there is $p$ such that $\mc{M}$ is a $p$-scale of meshes for $M$.
\end{definition}

The mesh, by itself, is not sufficient to define piecewise polynomial spaces. We need, in addition, a choice of maps $\gamma_T$. 
\begin{definition}[Coordinates for a scale of meshes]
\label{d:coord}
Let $\mc{M}=(I,\{T_h\}_{h\in I})_{h\in I}$ be a $p$-scale of meshes for $M$, we call a collection $\{\gamma_T\}_{T\in \mc{T}_{h},h\in I}$ a \emph{set of coordinates for $\mc{M}$} if there is $R>0$ such that for all $h\in I$, $k\in (0,\infty)$, $\{\gamma_T\}_{T\in \mc{T}_{h}}$ is a $(p,R)$-regular set of coordinates for $\mc{T}_{h}$. 
\end{definition}

\begin{remark}
Although in practice, the coordinate mappings $\gamma_T$ are usually either affine or isoparametric (i.e. mappings whose coordinate functions are, themselves polynomials of some fixed degree, in many theoretical considerations one assumes, for example, that the boundary of a domain is perfectly resolved. Since this cannot be done with affine or isoparametric mappings, we retain the flexibility to have more general coordinates.   
\end{remark}
\begin{remark}
In order that an isoparametric mapping \blue{of some fixed degree $p$} be a set of coordinates for a scale of meshes it is necessary only that the coefficients of the polynomials involved be uniformly bounded as $h\to 0$ and that the inverse of the Jacobian of the mapping be uniformly bounded above as $h\to 0$.
\end{remark}

We next define spaces of piecewise polynomials on a scale of meshes. We emphasize again that this definition depends \emph{both} on the mesh and on the coordinates for the mesh.
\begin{definition}[Piecewise polynomial spaces]
Let $\mc{M}:=(I,\{\mc{T}_h\}_{h\in I})$ be a scale of meshes for $M$ and $\mc{C}:=\{\gamma_T\}_{T\in \mc{T}_h,h\in I}$ a set of coordinates for $\mc{M}$. Let $\someLetter$ be the reference element for $\mc{M}$, $p\in \{0,1,\dots\}$, and $m\geq 0$. For $h\in I$, we define the polynomial approximation space of degree $p$ by 
$$
\mc{S}_{\mc{M},\mc{C},h}^{p,m}:=\{ u\in L^2(M)\,:\, u\circ \gamma_T\in \mathbb{P}_p|_{h\someLetter}\}\cap H^m(M),
$$
where $\mathbb{P}_p$ denotes the space of polynomials of degree $p$ on $\mathbb{R}^d$. Let $P_{\mc{T}_h,\ell}^{p,m}:H_{\mc{T}_h}^{\ell}(M)\to \mc{S}_{\mc{M},\mc{C},h}^{p,m}$ denote the broken $H_{k,\mc{T}_h}^{\ell}(M)$ orthogonal projection onto $\mc{S}_{\mc{M},\mc{C},h}^{p,m}$; i.e. the orthogonal projector with respect to any inner product whose norm is equivalent to 
$$
\|u\|_{H_{k,\mc{T}_h}^\ell(M)}^2:=\|u\|_{L^2(M)}^2+\sum_{T\in\mc{T}_h}\langle k\rangle^{-2\ell}\|u\|_{H^{\ell}(T)}^2,\qquad \langle k\rangle:=(1+k^2)^{1/2}.
$$ 
\end{definition}
Observe that, for some combinations of $m$ and $p$, $\mc{S}_{\mc{M},\mc{C},h}^{p,m}$ may consist only of global polynomials of degree $p$. However, it is useful to consider the conforming spaces of polynomials when trying to understand the `frequency-$k$' part of the error. See Theorems~\ref{t:RdHs} and~\ref{t:MHs}. Although analogous statements hold in the broken spaces, these statements are weaker than the ones given; in particular, the dual space for $H_{k,\mc{T}}^\ell$ requires $L^2$-type regularity at the interfaces between mesh elements and so does not measure only the frequency $\sim k$ parts of a function. In practice, one typically uses piecewise polynomials which are at most in $H^1$ but, because it does not complicate the analysis, we retain the flexibility to take $m>1$. 
\begin{remark}
Notice that when $u\in H^\ell(M)$,
$$
\|u\|^2_{H_{k,\mc{T}_h}^\ell(M)}=\|u\|^2_{H_k^\ell(M)}:= \|u\|_{L^2(M)}^2+\langle k\rangle^{-2\ell}\|u\|_{H^\ell(M)}^2.
$$
\end{remark}

\begin{remark}
It is more standard to work with a fixed reference element, $\someLetter$, rather than the shrinking element $h\someLetter$. However, the latter will be more convenient here and one can translate between the methods by pre-composing each of our coordinate $\gamma_T$ with the scaling map $s_h:\someLetter\to h\someLetter$, $s_h(x)=hx$. Defining meshes as in Definitions~\ref{d:scale} and~\ref{d:coord} allows us to guarantee that certain estimates (e.g the Poincar\'e--Wirtinger inequality) can be made uniform as $h\to 0$. The assumptions needed to guarantee these uniform estimates could instead be encoded directly in the derivatives of coordinate maps $\tilde{\gamma}_{T}:\someLetter\to M$, but the necessary assumptions on the maps from a fixed domain to a small ($h$-size) domain would be more complicated than those we use (from a small domain to a small domain).
\end{remark}

\subsection{Lower bounds for approximations on $\mathbb{R}^d$}
Although we give applications to meshes on manifolds below, our results are simplest to understand when approximating functions on $\mathbb{R}^d$ and we state them in this case first. 
For $u\in L^2(\mathbb{R}^d)$, we let $\hat{u}$ denote its Fourier transform. 
\begin{theorem}
\label{t:RdL2}
\blue{Let $r\in \{0,1,\dots,\}\cup \{\infty\}$, $\mc{M}=(I,\{\mc{T}_h\}_h)$ a $2(r+1)$-scale of meshes for $\mathbb{R}^d$, $\mc{C}$ a set of coordinates for $\mc{M}$, and $0<\Xi_L<\Xi_H$. } Then there are $c>0$ and $k_0>0$ such that for all \blue{$0\leq p\leq r$, $0\leq \ell\leq m\leq p+1$,all $k>k_0c^p$,} all $u\in L^2(\mathbb{R}^d)$ satisfying
\begin{equation}
\label{e:FourierSupport}
\supp \hat{u}\subset \{\xi\in \mathbb{R}^d\,:\,\Xi_L k\leq |\xi|\},\qquad \|u\|_{H^{2(p+1)}(\mathbb{R}^d)}\leq \blue{\langle \Xi_Hk+p+1\rangle^{2(p+1)}}\|u\|_{L^2(\mathbb{R}^d)},
\end{equation}
 all $h\in I$ with \blue{$chk/p^2\leq 1$}, and all  $0\leq m',m$ we have
\begin{equation}
\label{e:lowL2Rd}
\blue{c^p\Big(\frac{hk}{p^2})^{p+1-m'}}\|u\|_{L^2(\mathbb{R}^d)}\leq \|(I-P_{\mc{T}_h,\ell}^{p,m})u\|_{H_{k,\mc{T}_h}^{m'}(\mathbb{R}^d)}.
\end{equation}
Furthermore, if $p=0$, then $k_0$ can be taken arbitrarily small. 
\end{theorem}

\begin{remark}
One should heuristically understand~\eqref{e:FourierSupport} as follows. The first condition guarantees that $u$ has no very low frequencies ($\ll k$), while the second guarantees that it has no very high frequencies ($\gg k$). 
\end{remark}
It is easy to see that Theorem~\ref{t:RdL2} is optimal \blue{for uniformly bounded $p$}. Indeed, any $u$ satisfying~\eqref{e:FourierSupport} has
$$
\|\partial_x^\alpha u\|_{L^2(\mathbb{R}^d)}\leq C\Xi_H^{|\alpha|}\langle k\rangle ^{|\alpha|}\|u\|_{L^2(\mathbb{R}^d)}, \qquad |\alpha|\leq 2(p+1),
$$
for some $C$ depending only on $\alpha$ and the choice of norm on $H^s$.
and hence the standard estimate 
$$
\|(I-P_{\mc{T}_h,m}^{p,m})u\|_{H^{m}(\mathbb{R}^d)}\leq Ch^{p+1-m}\|u\|_{H^{p+1}(\mathbb{R}^d)}\qquad 0\leq m\leq p+1,
$$
(see e.g~\cite[Theorem 4.3.19]{SaSc:11},~\cite[Section 4.4]{Bre:08},~\cite[Section 3.1]{Ci:02}) together with the fact that our $u$ satisfies
$$\|u\|_{H^{s}(\mathbb{R}^d)}\leq  C\langle k\rangle^{s}\|u\|_{L^2(\mathbb{R}^d)},\qquad 0\leq s\leq p+1$$
 shows that, up to a constant,~\eqref{e:lowL2Rd} cannot be improved for many standard scales of meshes. Formally the results in these references apply only to compact subsets of $\mathbb{R}^d$, however, since the $H^m$ norm is local in the sense that for any open domains $\someLetter_i$ with Lipschitz boundary and
 $$
 \bigcup_i \overline{\someLetter}_i = \mathbb{R}^d,\, \someLetter_i\cap \someLetter_j=\emptyset, \text{ for }i\neq j,
 $$
 we have
 $$
 \|u\|_{H^m(\mathbb{R}^d)}^2=\sum_{\someLetter_i}\|u\|_{H^m(\someLetter_i)}^2,\qquad u\in H^m(\mathbb{R}^d),\,  $$
  these results easily extend to all of $\mathbb{R}^d$. 
\begin{remark}
Note that, while we write the estimate~\eqref{e:lowL2Rd} with the $L^2$ norm of $u$ on the left hand side, we could replace it by the $H^{p+1}(\mathbb{R}^d)$ norm and an appropriate power of $k$ using the second inequality in~\eqref{e:FourierSupport}. \blue{In fact, the proofs below proceed by controlling the $H_k^{p+1}$ norm of $u$ by  e.g. $\|(I-P_{\mc{T}_h,\ell}^{p,m})u\|_{L^2(\mathbb{R}^d)}$, and a small multiple of the $H_k^{2(p+1)}$ norm of $u$. The oscillatory nature of $u$ is then used to absorb this very high derivative into the left-hand side.}
\end{remark}

In particular, Theorem~\ref{t:RdL2} is optimal when the mesh is conforming in the following sense.
\begin{definition}
We say that the mesh $\mc{T}_h$ is \emph{$(p,m)$ conforming} if
\begin{equation}
\label{e:conforming}
\|I-P_{\mc{T}_h,\ell}^{p,m}\|_{H_k^{p+1}\to H_{k,\blue{\mc{T}_h}}^{\ell}}\leq C(hk)^{p+1-\ell}
\end{equation}
\end{definition}

It is often interesting not only to have lower bounds for the approximation error in $H_{k}^s$, but to understand lower bounds for the `frequency $k$' components of the best $H_{k}^s$ approximant. This is the content of our next theorem.
\begin{theorem}
\label{t:RdHs}
\blue{Let $r\in \{0,1,\dots\}\cup \{\infty\}$, $\mc{M}=(I,\{\mc{T}_h\}_h)$ be a $2(r+1)$-scale of meshes for $\mathbb{R}^d$, $\mc{C}$ be a set of coordinates for $\mc{M}$, and $0<\Xi_L<\Xi_H$.} Then there are $c>0$ and $k_0>0$ such that for all \blue{$0\leq p\leq r$, $0\leq \ell \leq m\leq p+1$, $s\geq 0$, all $k>k_0c^{-p}$}, all $u\in L^2(\mathbb{R}^d)$ satisfying
\begin{gather*}
\supp \hat{u}\subset \{\xi\in \mathbb{R}^d\,:\,\Xi_L k\leq |\xi|\},\\ \|u\|_{H_k^{\max(2(p+1),2\ell+s)}(\mathbb{R}^d)}\leq \blue{\langle \Xi_H k+\max(2(p+1),2\ell+s)\rangle}^{\max(2(p+1),2\ell+s)}\|u\|_{L^2(\mathbb{R}^d)}
\end{gather*}
 all \blue{$h\in \mc{I}$ with $chk/p^2\leq 1$} we have
 \begin{equation}
 \label{e:lowerHs}
\blue{c^p\Big(\frac{hk}{p^2}\Big)^{2(p+1-\ell)}}\|u\|_{L^2(\mathbb{R}^d)}\leq \|(I-P_{\mc{T}_h,\ell}^{p,m})u\|_{H_k^{-s}(\mathbb{R}^d)}.
\end{equation}
If $p=0$, then $k_0$ can be taken arbitrarily small. Finally, the estimate~\eqref{e:lowerHs} is optimal for \blue{bounded $p$} and meshes which are $(p,m)$-conforming.
\end{theorem}
Because the $H_k^{-s}$ norm weights frequencies $|\xi|\gg k$ by $|k^{-1}\xi|^{-s}$, Theorem~\ref{t:RdHs} shows that the `frequency $k$' components of the error are in general much smaller than the very high frequency components of the error (note that the power on the left hand side of~\eqref{e:lowerHs} is twice that on the left-hand side of~\eqref{e:lowL2Rd}), but nevertheless retain a controllable amount of the mass of $u$.

\subsection{Functions oscillating with a given frequency on a manifold}
\label{s:oscDef}
In order to state our results on a manifold, we first introduce an appropriate notion of a function that oscillates at frequency $k$ in a certain Sobolev space $H^m$.
\begin{definition}
\label{d:osc}
 Let $m\geq 1$, $a\leq b$,  $M$ be a $C^{m}$ manifold with Lipschitz boundary and $g\in C^1$ be a Riemannian metric on $M$. Let $-\Delta_g:L^2(M)\to L^2(M)$ denote the Dirichlet or Neumann Laplace--Beltrami operator on $(M,g)$ (i.e. the Friedrichs extension defined by the quadratic form $Q(u,v):=\langle \nabla_gu,\nabla _gv\rangle_{L^2(M)}$ with form domain $H_0^1$ or $H^1$ respectively) and $dE_\lambda$ its spectral measure.

 \begin{remark}
 As discussed above, in finite and boundary element methods the relevant manifold $(M,g)$ is usually either a domain $\mathscr{O}\subset \mathbb{R}^n$ with boundary or a hyper-surface $\Gamma\Subset \mathbb{R}^n$. In the case of a domain $\mathscr{O}$, we typically use the usual Euclidean metric, in which cases $-\Delta_g$ is the standard Dirichlet or Neumann Laplacian on $\mathscr{O}$. In the case of a surface $\Gamma$, we endow $\Gamma$ with the metric induced from the Euclidean metric on $\mathbb{R}^n$. That is, we the usual surface volume. In this case, once again, $-\Delta_g$ becomes the surface Laplacian (with the relevant boundary condition). 
 \end{remark}

We say that $u\in L^2(M)$ \emph{oscillates with frequencies between $a$ and $b$ in $H^m$} if  \blue{$u\in \mc{D}((-\Delta_g+1)^{m/2})$ and}
$$
\Pi_{[a,\infty)} u=u, \qquad \|u\|_{H^s(M)}\leq \langle b+\blue{j}\rangle^{\blue{j}}\|u\|_{L^2(M)},\,\,0\leq \blue{j}\leq m
$$
where we write
$$\Pi_{[a,\infty)}:=\int_{a^2}^{\infty} dE_\lambda$$
for the orthogonal projection onto functions oscillating with frequencies larger than $a$
\end{definition}

\begin{remark}
\blue{Below, we use the notation $C^\omega$ to denote a real analytic object and  $j<\omega$ for all $j\in\mathbb{N}$. For the purposes of scales of meshes and coordinates, we identify $\omega$ and $\infty$.}
\end{remark}

\noindent{\bf Examples:}
\begin{enumerate}
\item If $(M,g)$ is a compact manifold without boundary, then $-\Delta_g$ has an orthonormal basis of eigenfunctions $\{u_{\lambda_j}\}_{j=1}^\infty$ satisfying $(-\Delta_g-\lambda_j^2)u_{\lambda_j}=0$ and hence 
$$
\Pi_{[a,\infty)}v=\sum_{\lambda_j\in[a,\infty)}\langle v,u_{\lambda_j}\rangle_{L^2(M)} u_{\lambda_j}.
$$
\item If $(M,g)=(\mathbb{R}^d,g_{\text{Euc}})$ is $\mathbb{R}^d$ with the standard metric, 
$$
\Pi_{[a,\infty)}u= \frac{1}{(2\pi )^d}\int_{a\leq |\xi|} e^{i\langle x,\xi\rangle}\hat{u}(\xi)d\xi.
$$
\end{enumerate}

It will also be convenient to have a notion of approximately $k$ oscillating.
\begin{definition}
Let $\{C_j\}_{j=1}^\infty\subset \mathbb{R}^+$. We say that a family of functions $\{u_k\}_k\in L^2(M)$ is approximately $k$ oscillating with constants $C_j$ if for all $j=0,1,\dots$, and $k>1$,
$$
\|\Pi_{[\e k,\infty)}u_k-u_k\|_{L^2}\leq C_j k^{-j},\qquad \|u_k\|_{H_k^j(M)}\leq C_j\|u_k\|_{L^2(M)}
$$
\end{definition}

\subsection{Approximate $k$-oscillation and solutions of the Helmholtz equation}
\label{s:approx-k}

The main motivation for this article is the study of numerical solution of the Helmholtz problems~\eqref{e:inHome} and~\eqref{e:obstacle} when the data comes from a natural scattering problem; e.g. plane wave scattering. In the case of plane wave scattering, one aims to find the scattered field caused by an incident plane wave $u_I:=e^{ikx\cdot a}$ with $a\in \mathbb{R}^d$, $|a|=1$. To do this, we let $\chi\in C_c^\infty(\mathbb{R}^d)$ so that the scatterer is contained in $\chi\equiv 1$ and find $u_S$ outgoing such that $u=u_S+u_I$ satisfies~\eqref{e:inHome} or~\eqref{e:obstacle} with $0$ right-hand side and no outgoing condition 

 In this case, $u_S$ solves~\eqref{e:inHome} or~\eqref{e:obstacle} with 
\begin{equation}
\label{e:plane}
f=(k\chi_1(x)+\chi_2(x))e^{ikx\cdot a},\qquad \chi_i\in C_c^\infty(\mathbb{R}^d),
\end{equation}
and
\begin{equation}
\label{e:plane2}
g= \phi(x)e^{ikx\cdot a},\qquad \phi\in C^\infty(\partial\Omega).
\end{equation}

\subsubsection{$k$-oscillation of the bulk solution}
Using methods of semiclassical analysis; specifically the elliptic parametrix construction (see e.g.~\cite[Appendix E]{ZwScat}), one can show that for $a^{ij},c\in C^\infty(\mathbb{R}^d)$ with $c(x)>c_0>0$ and $a^{ij}(x)\xi_i\xi_j\geq c_0|\xi|^2$, the solution, $u$ to~\eqref{e:inHome} with $f$ of the form~\eqref{e:plane} is approximately $k$ oscillating for some constants $C_j$. 
Furthermore, for obstacle scattering, when the boundary of the obstacle is smooth and the data is as in~\eqref{e:plane} and~\eqref{e:plane2} one can use the functional calculus techniques from~\cite{GaLaSpWu:21} to see that the solution to the Helmholtz equation~\eqref{e:obstacle} is approximately $k$ oscillating.

The estimates in Theorems~\ref{t:RdL2},~\ref{t:RdHs} (above) and~\ref{t:ML2},~\ref{t:MHs} (below) then have direct applications to error analysis in finite element methods (FEM) based on piecewise polynomials. For example, when the FEM using the space $\mc{S}^{p,m}_{\mc{M},\mc{C},h}$ is applied to solve one of~\eqref{e:inHome} or~\eqref{e:obstacle} a key role in this analysis is played by the quantity
$$
\eta\big(\mc{S}^{p,m}_{\mc{M},\mc{C},h}\big):=\sup_{f\in L^2}\frac{\|(I-P_{\mc{T}_h,1}^{p,m})u_f\|_{H^1}}{\|f\|_{L^2}},
$$
where $u_f$ is the solution to~\eqref{e:inHome} or~\eqref{e:obstacle} with the radiation condition changed to
$$
 (\partial_r+ik)u=o_{r\to \infty}(r^{\frac{1-d}{2}}).
$$
Indeed, conditions for quasioptimality of FEM as well as error estimates are given explicitly in terms of this $\eta$~\cite{Sc:74, Sa:06,MeSa:10}. Because the solution of the Helmholtz problem is approximately $k$-oscillating, Theorems~\ref{t:RdL2} and~\ref{t:ML2} thus give sharp lower bounds on this quantity and hence provide lower estimates on how refined the grid must be to apply these results.

\subsubsection{$k$-oscillation of the boundary traces} In the case of the boundary element method, one tries to approximate the boundary traces of $u$ rather than $u$ itself. The $k$-oscillation properties of the traces of $u$ are slightly more complicated than those of $u$. It is easy to see that the upper bounds in~\eqref{d:osc} hold. However, the lower frequency bound may not hold. Nevertheless, the boundary traces will typically have a component containing most of the energy of the trace which is approximately $k$ oscillating. In fact, the only way for this to fail is for the function, $u$ to have nearly all of its energy on directions normal to the boundary. As an illustrative example, consider a smooth, convex obstacle $\Omega$ and the plane wave data as in~\eqref{e:plane2}. Then, one can show that the outgoing Dirichlet to Neumann map is pseudodifferential (albeit in an exotic calculus) and hence that $k^{-1}\partial_\nu u \sim a(x)e^{ikx\cdot a}|_{\partial\Omega}$. In particular, except in a neighborhood of $\nu\cdot a=0$, $k^{-1}\partial_\nu u$ is approximately $k$-oscillating. This type of argument can be made more precise and more general by using the concepts of wavefront-set and propagation of singularities from semiclassical analysis (see~\cite[Appendix E]{ZwScat}).  

\subsection{Lower bounds on a manifold}

We now restate Theorems~\ref{t:RdL2} and~\ref{t:RdHs}, generalizing them to Riemannian manifolds in the process.
\begin{theorem}
\label{t:ML2}
\blue{Let $r\in \{0,1,\dots\}\cup \{\omega\}$, $M$ be a $C^{2(r+1)}$ manifold with Lipschitz boundary, $g$ be  a $C^{r+1}$ Riemannian metric on $M$, $\mc{M}=(I,\{\mc{T}_h\}_h)$ be a $2(r+1)$-scale of meshes for $M$, $\mc{C}$ be a set of coordinates for $\mc{M}$, and  $0<\Xi_L<\Xi_H$.} Then there are $c>0$ and $k_0>0$ such that for all \blue{$0\leq p\leq r$,} $0\leq\ell\leq  m\leq p+1$,  , all  \blue{$k>k_0c^{-p}$,} all $u\in L^2(M)$ oscillating with frequencies between $\Xi_Lk$ and $\Xi_Hk$ in $H^{2(p+1)}(M)$, $0\leq m'$, and all $h\in I$ with \blue{$chk/p^2<1$,}  we have
\begin{equation}
\label{e:lowL2M}
\blue{c^p\Big(\frac{hk}{p^2}\Big)^{p+1-m'}\|u\|_{L^2(M)}}\leq \|(I-P_{\mc{T}_h,\ell}^{p,m})u\|_{H_{k,\mc{T}_h}^{m'}(M)}.
\end{equation}
Furthermore, if $p=0$, then $k_0$ can be taken arbitrarily small. 
\end{theorem}

As in $\mathbb{R}^d$, we also obtain lower bounds for the `frequency $k$' part of the error.
\begin{theorem}
\label{t:MHs}
\blue{Let $r\in \{0,1,\dots\}\cup \{\omega\}$, $M$ be a $C^{2(r+1)}$ manifold with Lipschitz boundary, $g$ be  a $C^{r+1}$ Riemannian metric on $M$, $\mc{M}=(I,\{\mc{T}_h\}_h)$ be a $2(r+1)$-scale of meshes for $M$, $\mc{C}$ be a set of coordinates for $\mc{M}$, and  $0<\Xi_L<\Xi_H$.}  Then there are $c>0$ and $k_0>0$ such that for all \blue{$0\leq p\leq r$, $s\leq 2(p+1)$, $0\leq \ell\leq (r+1)-\frac{s}{2}$, $\ell\leq m\leq p+1$,} all $k>k_0c^{-p}$, all $u\in L^2(M)$ oscillating with frequencies between $\Xi_Lk$ and $\Xi_Hk$  in $H^{\max(2(p+1),2\ell+s)}(M)$ and all \blue{$h\in I$ with $chk/p^2<1$},we have
\begin{equation}
\label{e:lowerHs2}
\blue{c^p\Big(\frac{hk}{p^2}\Big)^{2(p+1-\ell)}}\|u\|_{L^2(M)}\leq \big\|(I-P_{\mc{T}_h,\ell}^{p,m})u\big\|_{H_k^{-s}(M)}.
\end{equation}
Furthermore, if $p=0$, then $k_0$ can be taken arbitrarily small. Finally, the for $p$ bounded uniformly in $k$, estimate~\eqref{e:lowerHs2} is optimal if the boundary of $M$ is $C^{\max(2(p+1),2\ell+s)}$ with the mesh $(p,m)$-conforming.
\end{theorem}
Note that piecewise polynomial spaces which satisfy~\eqref{e:conforming} under various conditions on $\ell$ are constructed in~\cite[Chapter 4]{SaSc:11}.
\begin{remark}
In fact, if $\mc{C}$ consists only of affine maps, then one can take $k_0$ arbitrarily small for all $p$ in Theorems~\ref{t:RdL2} and~\ref{t:RdHs}. In general, when $p\neq 0$ and the maps $\gamma_T$ need not be affine, this is not possible. To see this,  we work on the circle $\mathbb{S}^1=[-\pi/2,3\pi/2)$. We need only consider a single mesh $\mc{T}:=\{ T_1,T_2,T_3,T_4\}$, $T_1:=(-\pi/2,0)$, $T_2:=(0,\pi/2)$, $T_3:=(\pi/2,\pi)$, $T_4:=(\pi,3\pi/2)\}$, with reference domain $[0,1]$. To define our coordinates, we will need two branches of $\sin^{-1}(t)$. For this, let $s_1:[-\pi/2,\pi/2]\to [-1,1]$, $s_1(t)=\sin (t)$, and $s_2:[\pi/2,3\pi/2]\to [-1,1]$, $s_2(t)=\sin (t)$. Set $\gamma_1(t)=s_1^{-1}(-1+t^2)$, $\gamma_2(t):=s_1^{-1}(1-t^2)$, $\gamma_3(t):=s_2^{-1}(1-t^2)$, and $\gamma_4(t):=s_2^{-1}(-1+t^2)$. 

To see that $\gamma_1$ is a regular coordinate map, observe that
$$
\gamma_1'(t)=\frac{-2t}{\sqrt{1-(1-t^2)^2}}=-\frac{2}{\sqrt{2-t^2}}.
$$
In particular, $\gamma_1'(t)$ is smooth up to the boundary of $(0,1)$ and satisfies $\gamma_1'(t)>c>0$. Similar analysis shows that $\gamma_i(t)$ is regular for $i=2,3,4$. 
Now, notice that 
\begin{gather*}
\sin(\gamma_1(t))=\sin(s_1(-1+t^2))=-1+t^2,\qquad \sin(\gamma_2(t))=\sin(s_1(1-t^2))=1-t^2,\\
 \sin(\gamma_3(t))=\sin(s_2(1-t^2))=1-t^2,\qquad \sin(\gamma_4(t))=\sin(s_2(-1+t^2))=-1+t^2.
\end{gather*}
In particular, $u:=\sin(x)\in \mc{S}_{\mc{M},\mc{C},1}^{2,2}$ so that $(I-P_{\mc{T}_1,\ell}^{2,2})u=0$. On the other hand, since $-\Delta_{\mathbb{S}}u=-\partial_x^2u=u$, we have $\Pi_{[1,\infty)}u=u$ and hence $u$ is oscillating with frequencies between $1$ and $2$. In particular, for this mesh, we do not have~\eqref{e:lowL2M} for functions oscillating with small frequency.
\end{remark}

Finally, we record an estimate when $u$ is approximately $k$-oscillating.
\begin{corollary}
Let  $p,s\geq 0$, $0\leq\ell\leq  m\leq p+1$, $0<\e<1$, and $\{C_j\}_{j=1}^\infty\subset \mathbb{R}_+$, $M$ be a $C^{\max(2(p+1),2\ell+s)}$ manifold with Lipschitz boundary and $g$ be a $C^{p+1}$ Riemannian metric on $M$. Let $\mc{M}$ be a $2(p+1)$ scale of meshes for $M$ and $\mc{C}$ be a set of coordinates for $\mc{M}$. Then for all $N>0$, there is $c>0$ such that for all $\e$ approximately $k$ oscillating, $u$ with constants $C_j$
, there is $k_0\geq 0$ such that for $k>k_0$, $0\leq m'\leq m$, and $h\in I$ with $h>k^{-N}$, we have
$$
c(hk)^{2(p+1-\ell)}\|u\|_{L^2(M)}\leq \|(I-P_{\mc{T}_h,\ell}^{p,m})u\|_{H_k^{-s}(M)},\qquad c(hk)^{p+1-m'}\|u\|_{L^2(M)}\leq \|(I-P_{\mc{T}_h,\ell}^{p,m})u\|_{H_k^{m'}(M)}.
$$
\end{corollary}

\subsection{Ideas behind the proof of Theorem~\ref{t:ML2}}
For the purposes of this outline, we work on $\mathbb{R}^d$, assume that $\gamma_T:\someLetter_h\to T$ is a rotation followed by a translation, and consider only $m'=0$. There are four important facts used to prove Theorem~\ref{t:ML2}:
\begin{enumerate}
\item For a function oscillating between $\Xi_L k$ and $\Xi_Hk$ in $H^{2(p+1)}$ and $p+1=2m+r$,
\begin{equation}
\label{e:outOscillating}
\begin{gathered}
ck^{2(p+1)}\|u\|_{L^2(\mathbb{R}^d)}^2\leq \langle (-\Delta)^{p+1}u,u\rangle_{L^2(\mathbb{R}^d)} = \|\nabla^r (-\Delta)^mu\|_{L^2(\mathbb{R}^d)}^2,\\
\|u\|^2_{H^{2(p+1)}(\mathbb{R}^d)}\leq C\langle k\rangle^{4(p+1)}\|u\|^{2}_{L^2(\mathbb{R}^d)}.
\end{gathered}
\end{equation}
\item We have
\begin{equation}
\label{e:outDecompose}
\|\nabla^r (-\Delta)^mu\|_{L^2(\mathbb{R}^d)}^2=\sum_{T\in \mc{T}_h}\|\nabla^r (-\Delta)^mu\|_{L^2(T)}^2.
\end{equation}
\item For a polynomial, $q_T$, of degree $p$ on $T$
\begin{equation}
\label{e:outInsertPoly}
\|\nabla^r (-\Delta)^mu\|_{L^2(T)}^2=\langle \nabla^r (-\Delta)^mu, \nabla^r(-\Delta)^m(u-q_T)\rangle_{L^2(T)}.
\end{equation}
\item Integrating by parts and using trace estimates, the pairings can be estimated 
\begin{multline}
\label{e:outPairing}
\sum_{T\in \mc{T}_h}|\langle \nabla^r (-\Delta)^mu, \nabla^r(-\Delta)^m(u-q_T)\rangle_{L^2(T)}|\\
\leq \e \|u\|^2_{H^{2(p+1)}(\mathbb{R}^d)}+C\e^{-1} h^{-2(p+1)}\|u-\sum_{T\in \mc{T}_h}1_Tq_T\|^2_{L^2(\mathbb{R}^d)}.
\end{multline}
\end{enumerate}
Here $1_T$ denotes the indicator function of $T$.
Combining~\eqref{e:outOscillating},~\eqref{e:outDecompose},~\eqref{e:outInsertPoly}, and~\eqref{e:outPairing} and choosing $\e= \e_0k^{-2(p+1)}$ for some $\e_0>0$ then completes the proof. 

The estimates~\eqref{e:outOscillating} follow directly from the definition of oscillating between $\Xi_Lk$ and $\Xi_Hk$, and~\eqref{e:outDecompose} follows from the definition of the $L^2$ norm. The equation~\eqref{e:outInsertPoly} follows from the fact that derivatives of order $\geq p+1$ vanish on a polynomial of order $p$. The work of this paper is then in proving~\eqref{e:outPairing}. This is done in two steps. First, in Section~\ref{s:reference}, we prove estimates on a pairing $\langle \partial_x^\alpha u,\partial_x^\alpha v\rangle_{L^2(T)}$ that are uniform in the scale $h$ and involve Sobolev norms of $u$ together with the $L^2$ norms of $v$ and its $(p+1)^{\text{th}}$ derivatives (see Lemma~\ref{l:elementEst}). We then combine the estimates on all elements of the mesh in Section~\ref{s:manifold} (see Lemma~\ref{l:upOut}) to obtain~\eqref{e:outPairing}.

\subsection{Comparison with Kolmogorov $n$-width bounds}
\label{s:nWidth}

The only other lower estimates on approximation by finite dimensional spaces of which the author is aware concern the $\mc{V}$-Komolgorov $n$-width of a space where $\mc{V}$ is a normed space (see~\cite{Je:72} and references there-in). For example, for $\Omega\subset \mathbb{R}^d$, the $L^2(\Omega)$-Komolgorov $n$ width of $\mc{B}\subset L^2(\Omega)$ is defined by 
$$
d_n(\mc{B}):=\sup_{u\in \mc{B},\,\|u\|_{\mc{B}}\leq1}\inf_{\substack{w\in \mc{W}\\\dim \mc{W}=n}}\|u-v\|_{L^2(\Omega)}.
$$
For instance,~\cite{Je:72} shows that when $\Omega$ has Lipschitz boundary, 
\begin{equation}
\label{e:k-nWidth}
0<\liminf_{n\to \infty} n^{\frac{1}{d}}d_n(H^1_0(\Omega))\leq \limsup_{n\to \infty}n^{\frac{1}{d}}d_n(H^1_0(\Omega))<\infty.
\end{equation}

For concreteness, we will consider the case of $H_0^1(\Omega)$ in the rest of this subsection. Standard \emph{upper} estimates on piecewise polynomial approximation then show that the space of piecewise polynomials saturate the estimate~\eqref{e:k-nWidth} in the sense that they achieve the estimate: for all $u\in H_0^1(\Omega)$,
\begin{equation}
\label{e:basicUpper}
\|(I-P_{\mc{T}_h,0}^{p,m})u\|_{L^2(\Omega)}\leq Ch\|u\|_{H^1(\Omega)}\leq C n^{-\frac{1}{d}}\|u\|_{H^1(\Omega)}.
\end{equation}
\begin{remark}
Here, we have used that a $p$-scale of meshes is necessarily quasiunifrom and hence the volume of any element is bounded above and below by $h^d$.
\end{remark}
 The estimate~\eqref{e:k-nWidth}, when applied to the space $\mc{S}_{\mc{M},\mc{C},h}^{p,m}$ shows that for $h$ small enough,
$$
\sup_{u\in H_0^1(\Omega),\|u\|_{H^1(\Omega)}\leq 1}\|(I-P_{\mc{T}_h,0}^{p,m})u\|_{L^2(\Omega)}\geq c h,
$$
and hence~\eqref{e:basicUpper} is optimal when one considers all possible $u$ in $H^1_0(\Omega)$.  In particular, there exists a function $u\in H_0^1(\Omega)$ such that $\|u\|_{H^1}\leq k$ and 
$$
\|(I-P_{\mc{T}_h,0}^{p,m})u\|_{L^2(\Omega)}\geq c hk.
$$
However, the estimate~\eqref{e:k-nWidth} gives no information about the structure of this $u$ nor how many such $u$ there are and hence it cannot be applied to understand approximation errors in concrete situations like Helmholtz scattering with natural data. 

In contrast, the estimates in Theorem~\ref{t:ML2} show that \emph{every} $k$-oscillating function with $\|u\|_{L^2}\sim 1$ (and hence $\|u\|_{H^1}\sim k$) satisfies
$$
\|(I-P_{\mc{T}_h,0}^{p,m})u\|_{L^2(\Omega)}\geq c hk.
$$
In particular, as noted in Remark~\ref{s:approx-k}, these estimates apply to \emph{every} Helmholtz scattering solution and hence can be used to understand approximation errors for numerical solution of Helmholtz scattering problems.

\medskip
\noindent{\textbf{Acknowledgements.}} The author would like to thank Euan Spence for many helpful discussions on the definitions of meshes and approximation spaces as well as the presentation of the article. Thanks also to Christoph Schwab for useful discussions on polynomial approximation and Th\'eophile Chaumont--Frelet for the reference to Kolmogorov $n$-width and subsequent discussions.  The author is grateful to the anonymous referees for their careful reading and helpful comments.   The author was partially supported by EPSRC Early Career Fellowship EP/V001760/1 and Standard Grant EP/V051636/1.

\section{Estimates on the reference element}
\label{s:reference}

\subsection{Estimates at a fixed scale}
In this section, we recall several standard estimates for functions in Sobolev spaces on Lipschitz domains and their boundary values. The first estimate can be found in~\cite[Theorem 1.5.1.10]{Gr:85}.
\begin{lemma}
\label{l:trace}
Let $\someLetter\Subset\mathbb{R}^d$ be open with Lipschitz boundary. Then there is $C>0$ such that for all $u\in H^1(\someLetter)$, and $0<\e<1$
\begin{equation*}
\|u\|_{L^2(\partial \someLetter)}\leq C(\e^{-1}\|u\|_{L^2(\someLetter)}+ \e\|\nabla u\|_{L^2(\someLetter)})
\end{equation*}
\end{lemma}

Next, we recall a useful, equivalent norm on $H^m(\someLetter)$ for $m\in \{0,1,\dots\}$.  
\begin{lemma}
\label{l:equivalent}
Let $\someLetter\Subset \mathbb{R}^d$ open with Lipschitz boundary. Then for all $m\in \mathbb{N}$, there is $C_m>0$ such that for all $u\in H^m(\someLetter)$, 
\begin{equation}
\label{e:mExplicity}
\|u\|_{H^m(\someLetter)}\leq C^m m^{2m}\|u\|_{L^2(\someLetter)}+C^m\sum_{|\gamma|=m}\|\partial_x^\gamma u\|_{L^2(\someLetter)} .
\end{equation}
\end{lemma}
\begin{proof}
To prove~\eqref{e:mExplicity}, we follow the same proof as the non-$m$ explicit proof (see e.g.~\cite[Theorem 5.2]{AdFo:03}), but require the inverse estimate:
\begin{equation}
\label{e:inverseH1}
\|v\|_{H^1(\someLetter)}\leq C_{\someLetter}p^2\|v\|_{L^2(\someLetter)},\qquad v\in\mathbb{P}_p.
\end{equation}
Since $\someLetter$ is open with Lipschitz boundary we may decompose $\someLetter\subset \bigcup_{j=1}^N \someLetter_j$, with $\someLetter_j\Subset\tilde{\someLetter}_j\subset \Omega$ and, up to a rotation and translation
$$
\tilde{\someLetter}_i:=\{(x_1,\dots, x_d)\in\mathbb{R}^d\,:\, |(x_1,\dots,x_{d-1})|<\e_i,\, 0\leq x_d\leq f_i(x_1,\dots,x_{d-1})\}
$$
where $f_i>c_i>0$ and $\|\nabla f\|_{L^\infty}\leq \tan (\theta_i)$

Now, let $\{\omega_j\}_{j=1,\dots, d}\subset \mathbb{R}^d$ be linearly independent, have unit length, and satisfy $\langle \omega_j,(0,\dots,0,1)\rangle> \cos (\theta_i).$ Then, for $v\in\mathbb{P}_p$, we have by rotating coordinates,
\begin{align*}
\int_{\someLetter_i} |\langle \nabla,\omega_j\rangle v|^2dx &= \int_{U_{j,i}}\int_0^{f_{i,j}(y')}|\partial_{y_d}v(y',y_d)|^2dy_ddy'\\
&\leq Cp^4\int_{U_{j,i}}\int_0^{f_{i,j}(y')}|v(y',y_d)|^2dy_{d}dy'=Cp^4\|v\|_{L^2(\someLetter_i)}^2,
\end{align*}
where in the last inequality, we have used that $v(y',\cdot)$ is a polynomial of degree $\leq p$, the inverse estimate~\cite[Theorem 3.91]{Sch:98}, and that $f_{i,j}>c>0$. 

Now, 
$$
|\nabla v|^2\leq C_i\sum_{j=1}^d|\langle \nabla v, \omega_j\rangle|^2.
$$
Hence, 
$$
\|\nabla v\|^2_{L^2(\someLetter_i)}\leq C_ip^4\|v\|^2_{L^2(\someLetter_i)}.
$$
Summing over $i$ completes the proof of~\eqref{e:inverseH1}.
\end{proof}
%
%

\subsection{Uniform estimates at all scales}
We now record the estimates corresponding to Lemma~\ref{l:trace} and Lemma~\ref{l:equivalent} on the rescaled domain $\someLetter_h:=h\someLetter$.
\begin{lemma}
Let $\someLetter\Subset \mathbb{R}^d$ open with Lipschitz boundary and $\someLetter_h:=h\someLetter$. There is $C>0$ such that for all $u\in H^1(\someLetter_h)$, $0<h<1$, and $0<\e<1$
\begin{equation}
\label{e:scaledTrace}
\|u\|_{L^2(\partial \someLetter_h)}\leq Ch^{-\frac{1}{2}}(\e^{-1}\|u\|_{L^2(\someLetter_h)}+ \e\|h\nabla u\|_{L^2(\someLetter_h)})
\end{equation}
\end{lemma}
\begin{proof}
Let $u\in H^1(\someLetter_h)$. Then, putting $v(x):=u(hx)\in H^1(\someLetter)$, we have
$$
\|v\|_{L^2(\partial\someLetter)}=h^{-\frac{d-1}{2}}\|u\|_{L^2(\partial \someLetter_h)},\quad \|v\|_{L^2(\someLetter)}=h^{-\frac{d}{2}}\|u\|_{L^2(\someLetter_h)},\quad \|\nabla v\|_{L^2(\someLetter)}=h^{-\frac{d}{2}}\|h\nabla u\|_{L^2(\someLetter_h)}. 
$$
The lemma now follows directly from Lemma~\ref{l:trace}.
\end{proof}
\begin{lemma}
Let $\someLetter\Subset \mathbb{R}^d$ open with Lipschitz boundary and $\someLetter_h:=h\someLetter$. There is $C>0$ such that for all $m\in \{0,1,\dots\}$. $u\in H^m(\someLetter)$, and $0<h<1$,
\begin{equation}
\label{e:scaledEllipticRegularity}
\blue{\|u\|_{H_h^m(\someLetter_h)}\leq C^mm^{2m}\|u\|_{L^2}+(Ch)^m\sum_{|\gamma|=m}\|\partial_x^\gamma u\|_{L^2(\someLetter_h)}.}
\end{equation}
\end{lemma}
\begin{proof}
Let $u\in H^{m}(\someLetter_h)$. Then $v(x):=u(hx)\in H^m(\someLetter)$ and the estimate follows from Lemma~\ref{l:equivalent} applied to $v$.
\end{proof}

\subsection{Estimates on pairings in $\someLetter_h$}

\begin{lemma}
\label{l:elementEst}
Let $\someLetter\Subset \mathbb{R}^d$ open with Lipschitz boundary. \blue{Then there is $C>0$ such that for all $p$,} $h>0$, $\alpha\in \mathbb{N}^d$ with $|\alpha|=p+1$, and $\someLetter_h:=h\someLetter$, there are $\beta_j\in \mathbb{N}^d$ with $|\beta_j|=p+1+j$, $j=0,1,\dots, p$ such that for all  $u,v\in H^{2(p+1)}(\someLetter_h)$, $\alpha_1+\alpha_2=\alpha$, $0<h<1$, and $0<r_j<1$
\begin{equation}
\begin{aligned}
\label{e:elementEst}
&|\langle \partial_{x}^{\alpha }u,\partial_{x}^{\alpha}v\rangle_{L^2(\someLetter_h)}|\leq \|\partial_{x}^{\alpha+\alpha_1}u\|_{L^2(\someLetter_h)}\|\partial_x^{\alpha_2}v\|_{L^2(\someLetter_h)}\\
&\qquad+\sum_{j=0}^{p-|\alpha_2|}C^{1+t-j}\| \partial_{x}^{\beta_j}u\|_{H_h^{1}( \someLetter_h)} \Big(h^{-t+j-1}s^{2(t-j)+1}\\
&\qquad \qquad+h^{-1-\frac{t-j}{p+2+j}}s^{1+\frac{t-j}{p+2+j}}r_j^{-\frac{t-j}{p+2+j}} +s^{-1-\frac{t-j+1}{p+1+j}}(Cr_j^{-1}) ^{\frac{t-j+1}{p+1+j}}\Big)| v|_{H^{|\alpha_2|}(\someLetter_h)}\\
&\qquad +\sum_{j=0}^{p-|\alpha_2|} r_jC^{t-j+1} \| \partial_{x}^{\beta_j}u\|_{H_h^{1}(\someLetter_h)}|v|_{H^{2(p+1)}(\someLetter_h)},
\end{aligned}
\end{equation}
where $s:=2(p+1)-|\alpha_2|$, $t:=p-|\alpha_2|$, and $|\cdot|_{H^m(\someLetter_h)}$ denotes the $H^m$ Sobolev seminorm.
\end{lemma}
\begin{proof}
Integration by parts shows that for $j=0,1,\dots, p-|\alpha_2|$ there are $\beta_j$, $\beta_j'$ with $|\beta_j|=p+1+j$, $|\beta_j'|=p-j$ and $f_j\in L^\infty(\partial \someLetter_h)$, \blue{with $\|f_j\|_{L^\infty}$ bounded by the Lipschitz constant of $\someLetter$} such that 
\begin{align*}
|\langle \partial_{x}^{\alpha }u,\partial_{x}^{\alpha}v\rangle_{L^2(\someLetter_h)}|&\leq |\langle \partial_{x}^{\alpha+\alpha_1}u,\partial_x^{\alpha_2}v\rangle_{L^2(\someLetter_h)}| +\sum_{j=0}^{p-|\alpha_2|}|\langle f_j \partial_{x}^{\beta_j}u, \partial_{x}^{\beta_j'}v\rangle_{L^2(\partial \someLetter_h)}|\\
&\leq \|\partial_{x}^{\alpha+\alpha_1}u\|_{L^2(\someLetter_h)}\|\partial_x^{\alpha_2}v\|_{L^2(\someLetter_h)}+\sum_{j=0}^{p-|\alpha_2|}\|f_j \partial_{x}^{\beta_j}u\|_{L^2(\partial \someLetter_h)}\| \partial_{x}^{\beta'_j}v\|_{L^2(\partial \someLetter_h)}.
\end{align*}

Then, using the Sobolev trace estimate~\eqref{e:scaledTrace} and the estimate~\eqref{e:scaledEllipticRegularity} on $\someLetter_h$,  together with interpolation in the $H_h^s(\someLetter_h)$ spaces, and Young's inequality, we have, with $t=p-|\alpha_2|$, 
\begingroup\blue{
\allowdisplaybreaks
\begin{align*}
&\leq\|\partial_{x}^{\alpha+\alpha_1}u\|_{L^2(\someLetter_h)}\|\partial_x^{\alpha_2}v\|_{L^2(\someLetter_h)}\\
&\qquad+\sum_{j=0}^{p-|\alpha_2|}\sum_{\substack{|\gamma|=p-j\\
|\gamma'|=p-j+1}}C\| \partial_{x}^{\beta_j}u\|_{H_h^{1}( \someLetter_h)} h^{-t+j-1}(\e_j^{-1}\|h^{t-j}\partial_x^\gamma v\|_{L^2( \someLetter_h)}+\frac{\e_j }{2} \|h^{t-j+1}\partial_x^{\gamma'} v\|_{L^2(\someLetter_h)})\\
&\leq\|\partial_{x}^{\alpha+\alpha_1}u\|_{L^2(\someLetter_h)}\|\partial_x^{\alpha_2}v\|_{L^2(\someLetter_h)}\\
&\qquad+\sum_{j=0}^{p-|\alpha_2|}\sum_{|\gamma|=|\alpha_2|}C\| \partial_{x}^{\beta_j}u\|_{H_h^{1}( \someLetter_h)} h^{-t+j-1}(\e_j^{-1}\|\partial_x^\gamma v\|_{H_h^{t-j}( \someLetter_h)}+\frac{\e_j }{2} \|\partial_x^{\gamma} v\|_{H_h^{t-j+1}(\someLetter_h)})\\
&\leq\|\partial_{x}^{\alpha+\alpha_1}u\|_{L^2(\someLetter_h)}\|\partial_x^{\alpha_2}v\|_{L^2(\someLetter_h)}\\
&\qquad+\sum_{j=0}^{p-|\alpha_2|}\sum_{|\gamma|=|\alpha_2|}C\| \partial_{x}^{\beta_j}u\|_{H_h^{1}( \someLetter_h)} h^{-t+j-1}\Big(\e_j^{-1}\|\partial_x^\gamma v\|_{L^2(\someLetter_h)}^{1-\frac{t-j}{2(p+1)-|\alpha_2|}}\|\partial_x^\gamma v\|^{\frac{t-j}{2(p+1)-|\alpha_2|}}_{H_h^{2(p+1)-|\alpha_2|}( \someLetter_h)}\\
&\qquad+\frac{\e_j }{2}\|\partial_x^\gamma v\|_{L^2(\someLetter_h)}^{1-\frac{t-j+1}{2(p+1)-|\alpha_2|}}\|\partial_x^\gamma v\|^{\frac{t-j+1}{2(p+1)-|\alpha_2|}}_{H_h^{2(p+1)-|\alpha_2|}( \someLetter_h)}\Big)\\
&\leq\|\partial_{x}^{\alpha+\alpha_1}u\|_{L^2(\someLetter_h)}\|\partial_x^{\alpha_2}v\|_{L^2(\someLetter_h)}\\
&\qquad+\sum_{j=0}^{p-|\alpha_2|}\sum_{|\gamma|=|\alpha_2|}C\| \partial_{x}^{\beta_j}u\|_{H_h^{1}( \someLetter_h)} h^{-t+j-1}\Big(\big(C^{t-j}\e_j^{-1}s^{2(t-j)}\|\partial_x^\gamma v\|_{L^2} \\
&\qquad \qquad+(Ch)^{t-j}\big[\e_j^{-1-\frac{t-j}{p+2+j}}(r_jh)^{-\frac{t-j}{p+2+j}}\|\partial_x^\gamma v\|_{L^2}+r_j h\sum_{|\tilde{\gamma}|=2(p+1)-|\alpha_2|}\|\partial_x^{\tilde{\gamma}+\gamma}v\|_{L^2}\big]\big)\\
&\qquad \qquad +\big(\e_jC^{t-j+1}s^{2(t-j+1)}\|\partial_x^\gamma v\|_{L^2} \\
&\qquad \qquad+(Ch)^{(t-j+1)(1+\frac{1}{p+1+j})}\big[\e_j^{1+\frac{t-j+1}{p+1+j}}(r_jh) ^{-\frac{t-j+1}{p+1+j}}\|\partial_x^\gamma v\|_{L^2}\big]\big)\Big)\\
&\leq\|\partial_{x}^{\alpha+\alpha_1}u\|_{L^2(\someLetter_h)}\|\partial_x^{\alpha_2}v\|_{L^2(\someLetter_h)}\\
&\qquad+\sum_{j=0}^{p-|\alpha_2|}C\| \partial_{x}^{\beta_j}u\|_{H_h^{1}( \someLetter_h)} \Big(h^{-t+j-1}C^{t-j}\e_j^{-1}s^{2(t-j)}\\
&\qquad \qquad+C^{t-j}h^{-1-\frac{t-j}{p+2+j}}\e_j^{-1-\frac{t-j}{p+2+j}}r_j^{-\frac{t-j}{p+2+j}} +\e_jh^{-t+j-1}C^{t-j+1}s^{2(t-j+1)} \\
&\qquad \qquad+C^{(t-j+1)(1+\frac{1}{p+1+j})}\e_j^{1+\frac{t-j+1}{p+1+j}}r_j ^{-\frac{t-j+1}{p+1+j}}\Big)| v|_{H^{|\alpha_2|}}\\
&\qquad+ \sum_{j=0}^{p-|\alpha_2|}\| \partial_{x}^{\beta_j}u\|_{H_h^{1}( \someLetter_h)} r_jC^{t-j+1} |v|_{H^{2(p+1)}}\\
&\leq\|\partial_{x}^{\alpha+\alpha_1}u\|_{L^2(\someLetter_h)}\|\partial_x^{\alpha_2}v\|_{L^2(\someLetter_h)}\\
&\qquad+\sum_{j=0}^{p-|\alpha_2|}C^{1+t-j}\| \partial_{x}^{\beta_j}u\|_{H_h^{1}( \someLetter_h)} \Big(h^{-t+j-1}s^{2(t-j)+1}\\
&\qquad \qquad+h^{-1-\frac{t-j}{p+2+j}}s^{1+\frac{t-j}{p+2+j}}r_j^{-\frac{t-j}{p+2+j}} +s^{-1-\frac{t-j+1}{p+1+j}}(Cr_j^{-1}) ^{\frac{t-j+1}{p+1+j}}\Big)| v|_{H^{|\alpha_2|}}\\
&\qquad +\sum_{j=0}^{p-|\alpha_2|}\| \partial_{x}^{\beta_j}u\|_{H_h^{1}( \someLetter_h)} r_jC^{t-j+1} |v|_{H^{2(p+1)}}
%
%
\end{align*}
where in the last line we grouped terms and set $\e_j=[2(p+1)-|\alpha_2|]^{-1}$, $s:=2(p+1)-|\alpha_2|$.}

\endgroup
\end{proof}

\section{Estimates on the manifold}
\label{s:manifold}

We now proceed to estimate the finite element approximation error. We first estimate a certain sum of derivatives over the mesh from below by the $L^2$ norm of $u$. 

\begin{lemma}
\label{l:lowOut}
\blue{Let $r\in \{0,1,\dots\}\cup \{\omega\}$, $M$ be a $C^{2(r+1)}$ manifold with Lipschitz boundary, $g$ be  a $C^{r+1}$ Riemannian metric on $M$, $\mc{M}=(I,\{\mc{T}_h\}_h)$ be a $2(r+1)$-scale of meshes for $M$, $\mc{C}$ be a set of coordinates for $\mc{M}$, and  $0<\Xi_L<\Xi_H$.}  Then there are $c>0$ and $k_0>0$ such that for all $u$ oscillating between $\Xi_Lk $ and $\Xi_Hk$ in $H^{p+1}$, $k>k_0c^{-p}$, and $h\in I$,
\begin{equation}
\label{e:lowOut}
\blue{c^{p+1}k^{2(p+1)}\|u\|_{L^2(M)}^2}\leq \sum_{T\in \mc{T}_h}\sum_{|\alpha|=p+1}\|\partial_x^\alpha (u\circ \gamma_{T})\|_{L^2(\someLetter_h,dx)}^2.
\end{equation}
Moreover, if $p=0$, then we may take $k_0=0$. 
\end{lemma}
\begin{proof}
Let $p+1=2m+r$ with $r\in \{0,1\}$, $m\in \{0,1,\dots\}$. Observe that 
\begin{equation}
\label{e:low1}
\Xi_L^{2(p+1)}k^{2(p+1)}\|u\|_{L^2(M)}\leq \langle (-\Delta_g)^{m+r}u,(-\Delta_g)^{m}u\rangle_{L^2(M)}=\|L_{g,p+1}u\|_{L^2(M)}^2
\end{equation}
where $L_{g,p+1}$ is a $p+1$ order differential operator with $L^\infty$ coefficients such that $1\in \text{ker}(L_{g,p+1})$ (i.e. $L_{g,p+1}$ has no constant term). Then
\begin{equation}
\label{e:low2}
\|L_{g,p+1}u\|_{L^2(M)}^2=\sum_{T\in \mc{T}_h}\|1_{\gamma_T(\someLetter_h)}L_{g,p+1}u\|_{L^2(M)}^2.
\end{equation}
Now, on each mesh element $\gamma_T(\someLetter_h)$, we write in coordinates
\begin{equation}
\label{e:diffOp}
L_{g,p+1}=\sum_{|\alpha|=p+1}a_{\alpha}^T\partial_x^\alpha+\sum_{1\leq |\beta|\leq p}b_\beta^T\partial_x^\beta.
\end{equation}
\blue{If $g$ is analytic, there there is $C>0$ such that for any $\alpha$ and $p$
$$\| a_{\alpha}^T\|_{L^\infty}\leq \|g^{-1}\|_{C^0}^{p+1} ,\qquad \|b_\beta^T\|_{L^\infty}\leq  C^{p+1}\prod_{\sum j\beta_j=p+1-|\beta|}\|g\|_{C^{j}}^{\beta_j}.$$
}
Therefore,
\begin{equation}
\label{e:low3}
\begin{aligned}
\frac{1}{2}\|1_{\gamma_T(\someLetter_h)}L_{g,p+1}u\|_{L^2(M))}^2&\leq \|\sum_{|\alpha|=p+1}a_{\alpha}^T\partial_x^\alpha (u\circ \gamma_T)\|_{L^2(\someLetter_h,d\textup{v}_g)}^2+\|\sum_{1\leq |\beta|\leq p}b_\beta^T\partial_x^\beta (u\circ\gamma_T)\|_{L^2(\someLetter_h,d\textup{v}_g)}^2\\
&\leq \blue{C_1^{p+1}} \sum_{|\alpha|=p+1}\|\partial_x^\alpha (u\circ\gamma_T)\|_{L^2(\someLetter_h,dx)}^2+\blue{C_2}\|d u\|_{H^{p-1}(\gamma_T(\someLetter_h))}^2.
\end{aligned}
\end{equation}
Summing over the mesh and using~\eqref{e:low1} and~\eqref{e:low2}, together with~\eqref{e:low3} we obtain
\begin{equation}
\label{e:low4}
\begin{aligned}
\Xi_L^{2(p+1)}k^{2(p+1)}\|u\|_{L^2(M)}&\leq \blue{C_1^{p+1}} \sum_{|\alpha|=p+1}\|\partial_x^\alpha (u\circ\gamma_T)\|_{L^2(\gamma_T(\someLetter_h),dx)}^2+\blue{C_2}\|d u\|_{H^{p-1}(M)}^2\\
&\leq \blue{C_1^{p+1}} \sum_{|\alpha|=p+1}\|\partial_x^\alpha (u\circ\gamma_T)\|_{L^2(\gamma_T(\someLetter_h),dx)}^2+\blue{C_2\langle \Xi_{H}k+p\rangle ^{2p}}\|u\|_{L^2(M)}^2.
\end{aligned}
\end{equation}
Here, $du$ is the differential of $u$.
Taking $k_0$ large enough, we may absorb the last term into the left-hand side and hence obtain the result for $p\geq 1$. For $p=0$, notice that the second term in~\eqref{e:diffOp} is identically 0 and hence there are no $\|du\|_{H^{p-1}}$ terms in~\eqref{e:low3} or~\eqref{e:low4}, so that we need not take $k_0$ large enough in this case.

\blue{In order to handle the analytic case, we need a slightly better estimate than~\eqref{e:low3}. Indeed, we estimate
$$
\begin{aligned}
\|\sum_{1\leq |\beta|\leq p}b_\beta^T\partial_x^\beta (u\circ\gamma_T)\|_{L^2(\someLetter_h,d\textup{v}_g)}&\leq C^{p+1} \sum_{1\leq \ell\leq p} \prod_{\sum j\beta_j=p+1-\ell} \|g\|_{C^j}^{\beta_j}\|du\|_{H^{\ell-1}}\\
&\leq C^{p+1}\sum_{1\leq \ell\leq p}(p+1-\ell)^{(p+1-\ell)}\langle \Xi_H k+\ell\rangle^{\ell}\|u\|_{L^2(M)}\\
&\leq C^{p+1}\langle \Xi_Hk+p\rangle^p\|u\|_{L^2(M)}.
\end{aligned}
$$
Thus, we require $k>C^{p+1}$ as stated.}

\end{proof}

Next, we estimate the right-hand side of~\eqref{e:lowOut} using the $L^2$ norm of $(I-P_{\mc{T}_h,\ell}^{p,m})u$. This lemma amounts to a type of Bramble--Hilbert Lemma for our polynomial spaces on a manifold.

\begin{lemma}
\label{l:upOut}
\blue{For all $0<\delta<1$, there is $C>0$  such that for all $0<h$,  $0\leq p,m\leq k$, $0<hk<\delta^{-1}p^2$, $0\leq m',\ell\leq m$, $\delta ^{p+1}<\e<1$, and $0<\Xi_L<\Xi_H$ and all $u$ oscillating between $\Xi_Lk$ and $\Xi_Hk$ in $H^{2(p+1)}$, } 
 \begin{multline}
 \label{e:upOut}\sum_{T\in \mc{T}_h}\sum_{|\alpha|=p+1}\|\partial_x^\alpha (u\circ \gamma_{T})\|_{L^2(\someLetter_h,dx)}^2\\\leq C^{\blue{p}}(\e \langle \Xi_Hk\rangle^{2(p+1)}\|u\|_{L^2(M)}^2 +\e^{-1}h^{-2(p+1-m')}(2(p+1)-m')^{4(p+1-m')}\|(I-P_{\mc{T}_h,\ell}^{p,m})u\|_{H^{m'}(M)}^2).
 \end{multline}
\end{lemma}
\begin{proof}
We start by observing that, since $[P_{\mc{T}_h,\ell}^{p,m}u]\circ\gamma_T$ is a polynomial of degree $p$, 
$$
\sum_{T\in \mc{T}_h}\sum_{|\alpha|=p+1}\|\partial_x^\alpha (u\circ \gamma_{T})\|_{L^2(\someLetter_h,dx)}^2=\sum_{T\in \mc{T}_h}\sum_{|\alpha|=p+1}\langle \partial_x^\alpha (u\circ \gamma_{T}), \partial_x^\alpha([(I-P_{\mc{T}_h,\ell}^{p,m})u]\circ \gamma_T)\rangle_{L^2(\someLetter_h)}.
$$
We apply Lemma~\ref{l:elementEst}, to each summand to obtain with $v=v_T:=[(I-P_{\mc{T}_h,\ell}^{p,m})u]\circ \gamma_T$, $u=u_T:=u\circ\gamma_T$. Note that we can do this since $\gamma_T\in C^{2(p+1)}$ and hence $u\circ\gamma_T\in H^{2(p+1)}$. We obtain with $t=p-m'$, $s:=2(p+1)-m'$,\blue{
\begin{align*}
&\sum_{T\in \mc{T}_h}\sum_{|\alpha|=p+1}\|\partial_x^\alpha u_T\|_{L^2(\someLetter_h,dx)}^2\\
&\leq \sum_{T\in \mc{T}_h}\sum_{\substack{|\alpha|=p+1\\
\alpha_1+\alpha_2=\alpha\\
|\alpha_2|=m'}}\Big(\|\partial_{x}^{\alpha+\alpha_1}u_T\|_{L^2(\someLetter_h)}\|\partial_x^{\alpha_2}v_T\|_{L^2(\someLetter_h)}+\sum_{j=0}^{p-m'}C^{t-j+1}\| \partial_{x}^{\beta_j}u_T\|_{H_h^{1}( \someLetter_h)} \Big(h^{-t+j-1}s^{2(t-j)+1}\\
&\qquad \qquad\qquad+h^{-1-\frac{t-j}{p+2+j}}s^{1+\frac{t-j}{p+2+j}}r_j^{-\frac{t-j}{p+2+j}}+s^{-1-\frac{t-j+1}{p+1+j}}(Cr_j^{-1}) ^{\frac{t-j+1}{p+1+j}}\Big)| v_T|_{m'}\\
&\qquad\qquad \qquad+\sum_{j=0}^{p-m'}\| \partial_{x}^{\beta_j}u_T\|_{H_h^{1}( \someLetter_h)} r_jC^{t-j+1} |v_T|_{H^{2(p+1)}}\Big)
\end{align*}}
Now, using again that $[P_{\mc{T}_h,\ell}^{p,m}u]\circ\gamma_T$ is a polynomial of degree $p$, we have $\partial_{x}^{\gamma'}v_T=\partial_{x}^{\gamma'}u_T$. Hence applying Young's inequality, we have
\begingroup
\allowdisplaybreaks
\blue{
\begin{align*}
&\sum_{T\in \mc{T}_h}\sum_{|\alpha|=p+1}\|\partial_x^\alpha u_T\|_{L^2(\someLetter_h,dx)}^2\\
&\leq C\sum_{T\in \mc{T}_h}\sum_{\substack{|\alpha|=p+1\\
\alpha_1+\alpha_2=\alpha\\
|\alpha_2|=m'}}\Big(\delta \|\partial_{x}^{\alpha+|\alpha_1|}u_T\|^2_{L^2(\someLetter_h)}+\delta^{-1}\|\partial_x^{\alpha_2}v_T\|^2_{L^2(\someLetter_h)}\\
&\qquad+C\sum_{j=0}^{p-m'}C^{t-j+1}\| \partial_{x}^{\beta_j}u_T\|_{H_h^{1}( \someLetter_h)} \Big(h^{-t+j-1}s^{2(t-j)+1}\\
&\qquad \qquad+h^{-1-\frac{t-j}{p+2+j}}s^{1+\frac{t-j}{p+2+j}}r_j^{-\frac{t-j}{p+2+j}} +s^{-1-\frac{t-j+1}{p+1+j}}(C r_j^{-1} )^{\frac{t-j+1}{p+1+j}}\Big)| v_T|_{m'}\\
&\qquad\qquad \sum_{j=0}^{p-m'}\| \partial_{x}^{\beta_j}u_T\|_{H_h^{1}( \someLetter_h)} r_jC^{t-j+1} |u_T|_{H^{2(p+1)}}\Big)\\
&\leq  \blue{C^{p+1}}\Big(\delta \|u\|^2_{H^{2(p+1)-m'}(M)}+\delta^{-1}\|(I-P_{\mc{T}_h,\ell}^{p,m})u\|^2_{H_{\mc{T}_h}^{m'}(M)}+\sum_{j=0}^{p-m'}(\delta_j(\| u\|^2_{H^{p+1+j}(M)}+h^2\| u\|^2_{H^{p+2+j}( M)})\\
&\qquad+\delta_j^{-1}\Big(h^{-2(t-j+1)}C^{2(t-j)}s^{4(t-j)+2}+C^{2(t-j)}h^{-2-\frac{2(t-j)}{p+2+j}}s^{2+2\frac{t-j}{p+2+j}}r_j^{-2\frac{t-j}{p+2+j}} \\
&\qquad\qquad+C^{2(t-j+1)(1+\frac{1}{p+1+j})}s^{-2-2\frac{t-j+1}{p+1+j}}r_j ^{-2\frac{t-j+1}{p+1+j}}\Big)\|(I-P_{\mc{T}_h,\ell}^{p,m})u\|^2_{H_{\mc{T}_h}^{m'}(M)}\\
&\qquad\qquad \sum_{j=0}^{p-m}\delta_j^{-1}r_j^2C^{2(t-j+1)} \|u\|^2_{H^{2(p+1)}(M)}\Big)\\
\end{align*}}
\blue{
Then, using that $u$ is oscillating between $\Xi_Lk$ and $\Xi_Hk$ and \blue{$p\leq C k$}, we have
\begin{align*}
&\sum_{T\in \mc{T}_h}\sum_{|\alpha|=p+1}\|\partial_x^\alpha u_T\|_{L^2(\someLetter_h,dx)}^2\\
&\leq  C^{\blue{p}}\Big(\Big(\delta \langle \blue{\Xi_Hk}\rangle^{4(p+1)-2m'}+\sum_{j=0}^{p-m'} \delta_j \langle \blue{\Xi_Hk}\rangle^{2(p+1+j)}\\
&\qquad+\delta_jh^2\langle \blue{\Xi_Hk}\rangle^{2(p+2+j)}+\delta_j^{-1}r_j^2C^{2(t-j+1)}\langle \blue{\Xi_Hk}\rangle^{4(p+1)}\Big)\|u\|^2_{L^2(M)} \\
&\qquad +\Big(\delta^{-1}+\sum_{j=0}^{p-m'}\delta_j^{-1}\Big(h^{-2(t-j+1)}C^{2(t-j)}s^{4(t-j)+2}\\
&\qquad \qquad+C^{2(t-j)}h^{-2-\frac{2(t-j)}{p+2+j}}s^{2+2\frac{t-j}{p+2+j}}r_j^{-2\frac{t-j}{p+2+j}} \\
&\qquad\qquad\qquad+C^{2(t-j+1)(1+\frac{1}{p+1+j})}s^{-2-2\frac{t-j+1}{p+1+j}}r_j ^{-2\frac{t-j+1}{p+1+j}}\Big)\Big)\|(I-P_{\mc{T}_h,\ell}^{p,m})u\|^2_{H_{\mc{T}_h}^{m'}(M)}\Big)\\
\end{align*}
Let 
\begin{gather*}
\delta=\e \langle \Xi_Hk\rangle^{-2(p+1)+2m'},\\
 \delta_j=\e \blue{\langle j\rangle^{-2}\langle \Xi_Hk\rangle^{-2j}}\, j=0,\dots, p-m' ,\\
 r_j=\e \langle \Xi_Hk\rangle^{-j-p-1}\, j=0,\dots, p-m' ,
\end{gather*}
then we obtain
\begin{align*}
&\sum_{T\in \mc{T}_h}\sum_{|\alpha|=p+1}\|\partial_x^\alpha u_T\|_{L^2(\someLetter_h,dx)}^2\\
&\leq  C^{\blue{p}}\Big(\Big(C\e  \langle\blue{\Xi_H} k\rangle ^{2(p+1)} +\e h^2\langle \blue{\Xi_H}k\rangle^{2(p+1)+2}+\e C^{2t}\langle\Xi_Hk\rangle^{2(p+1)}\Big)\|u\|^2_{L^2(M)} \\
&\qquad +\e ^{-1}\Big(\langle \blue{\Xi_Hk}\rangle^{2(p+1)-2m'} +C^{t}\Big(h^{-2(t+1)}s^{4t+2}+h^{-2-\frac{2t}{p+2}}s^{2+2\frac{t}{p+2}}\langle\Xi_H k\rangle^{2(1-\frac{1}{p+2})t}\e ^{-2\frac{t}{p+2}} \\
&\qquad\qquad+s^{-2-2\frac{t+1}{p+1}}\langle\Xi_H k\rangle^{2(t+1)}\e ^{-2\frac{t+1}{p+1}}\Big)\Big)\|(I-P_{\mc{T}_h,\ell}^{p,m})u\|^2_{H_{\mc{T}_h}^{m'}(M)}\Big).
\end{align*}}
\endgroup
\blue{Using $\e >\delta^{p+1}$,  and $hk/p^2<C$, we obtain the desired estimate.}
\end{proof}

\begin{proof}[Proof of the $L^2$ lower bound: Theorem~\ref{t:ML2}]
We now combine Lemmas~\ref{l:lowOut} and~\ref{l:upOut} to prove the main theorem. Indeed, Lemma~\ref{l:lowOut} implies that there are $k_0>0$ (with $k_0$ arbitrary when $p=0$) and $c_0>0$ such that for $k>k_0C^p$~\eqref{e:lowOut} holds. In particular, 
\begin{equation}
\label{e:final1}
c_0^pk^{2(p+1)}\|u\|_{L^2(M)}^2\leq \sum_{T\in \mc{T}_h}\sum_{|\alpha|=p+1}\|\partial_x^\alpha (u\circ\gamma_T)\|_{L^2(\someLetter_h,dx)}^2.
\end{equation}
Then, Lemma~\ref{l:upOut} implies that there is $C>0$ such that for $0\leq m'\leq m$,
\begin{equation}
\label{e:final2}
\begin{aligned}
&\sum_{T\in \mc{T}_h}\sum_{|\alpha|=p+1}\|\partial_x^\alpha (u\circ\gamma_T)\|_{L^2(\someLetter_h,dx)}^2\\
&\leq \frac{\blue{c^p}_0}{2}(1+k_0^{-2})^{p+1}\langle k^{2(p+1)}\|u\|_{L^2(M)}^2+\blue{C^p}h^{-2(p+1-m')}\blue{p^{4(p+1-m')}}\|(I-P_{\mc{T}_h,\ell}^{p,m})u\|_{H_{\mc{T}_h}^{m'}(M)}^2\\
&\leq \frac{\blue{c_0^p}}{2}k^{2(p+1)}\|u\|_{L^2(M)}^2+\blue{C^p}h^{-2(p+1-m')}\blue{p^{4(p+1-m')}}\|(I-P_{\mc{T}_h,\ell}^{p,m})u\|_{H_{\mc{T}_h}^{m'}(M)}^2.
\end{aligned}
\end{equation}
Combining~\eqref{e:final1} and~\eqref{e:final2}, we obtain
$$
\blue{c_0^p}k^{2(p+1)}\|u\|_{L^2(M)}^2\leq\frac{\blue{c_0^p}}{2}k^{2(p+1)}\|u\|_{L^2(M)}^2+\blue{C^p}h^{-2(p+1-m')}\blue{p^{4(p+1-m')}}\|(I-P_{\mc{T}_h,\ell}^{p,m})u\|_{H_{\mc{T}_h}^{m'}(M)}^2.
$$
Subtracting the first term on the right-hand side to the left-hand side, we obtain for $0\leq m'\leq m$
$$
\frac{\blue{c^p_0}}{2}k^{2(p+1)}\|u\|_{L^2(M)}^2\leq \blue{C^p}h^{-2(p+1-m')}\blue{p^{4(p+1-m')}}\|(I-P_{\mc{T}_h,\ell}^{p,m})u\|_{H_{\mc{T}_h}^{m'}(M)}^2,
$$
which completes the proof.
\end{proof}

\begin{proof}[Proof of the `frequency $k$' lower bound: Theorem~\ref{t:MHs}]
By Theorem~\ref{t:ML2}, we have
\begin{equation}
\label{e:L2Low}
\|(I-P_{\mc{T}_h,\ell}^{p,m})u\|_{H_k^\ell(M)}\geq \blue{c^p}(hk\blue{p^{-2}})^{p+1-\ell}\|u\|_{L^2(M)}.
\end{equation}
Next, since $\langle (I-P_{\mc{T}_h,\ell}^{p,m})u,P_{\mc{T}_h,\ell}^{p,m}u\rangle_{H_k^\ell(M)}=0$, and $u$ is oscillating with frequencies between $\Xi_Lk$ and $\Xi_Hk$ in $H^{\max(2(p+1),2\ell+s)}(M)$ we have
\begin{equation}
\label{e:L2toHs}
\begin{aligned}
&\|(I-P_{\mc{T}_h,\ell}^{p,m})u\|_{H_k^\ell(M)}^2=\langle (I-P_{\mc{T}_h,\ell}^{p,m})u,u\rangle_{H_k^\ell(M)}\\
&\leq\|(I-P_{\mc{T}_h,\ell}^{p,m})u\|_{H_k^{-s}(M)}\|u\|_{H_k^{2\ell+s}(M)}\\
&\leq C \| (I-P_{\mc{T}_h,\ell}^{p,m})u\|_{H_k^{-s}(M)}\|u\|_{L^2(M)}.
\end{aligned}
\end{equation}
Combining~\eqref{e:L2toHs} and~\eqref{e:L2Low} completes the proof.
\end{proof}

\section{Optimality of the low frequency bounds for shape regular meshes}

In this section, we show that Theorem~\ref{t:MHs} is optimal for any mesh satisfying~\eqref{e:conforming}. First, we prove the following elementary duality estimate based on the fact that $(I-P_{\mc{T}_h,\ell}^{p,m}):H_k^\ell\to H_k^\ell$ is an orthogonal projector. \begin{lemma}
\label{l:duality}
Suppose that $M$ is a $C^{2(p+1)}$ manifold with $C^{\max(2(p+1),2\ell+s)}$ boundary. Then for all $-\ell \leq s\leq 2(p+1)-\ell$, we have
$$
\|(I-P_{\mc{T}_h,\ell}^{p,m})\|_{H_k^\ell\to H_k^{-s}}\leq C\|(I-P_{\mc{T}_h,\ell}^{p,m})\|_{H_k^{2\ell+s}\to H_k^{\ell}}.
$$
\end{lemma}
\begin{proof}
Let $A_\ell:=(1-k^{-2}\Delta_g)^{\ell}$ for a $C^{\max(2(p+1),2\ell+s)}$ metric $g$. Then, $A_\ell$ is self-adjoint as an unbounded operator on $L^2$, $A_\ell:H_k^{t}\to H_k^{t-2\ell}$ is invertible for $0\leq t\leq \max(2(p+1),2\ell+s)$, and the inner product on $H_k^\ell(M)$ can be defined by 
$$
\langle v,w\rangle_{H_k^\ell}:=\langle  v,A_\ell w\rangle_{L^2}.
$$

 Then, since $(I-P_{\mc{T}_h,\ell}^{p,m})$ is orthogonal on $H_k^\ell$,
$$
\langle  (I-P_{\mc{T}_h,\ell}^{p,m})v,w\rangle_{H_k^\ell} =\langle v, (I-P_{\mc{T}_h,\ell}^{p,m})w\rangle_{H_k^\ell},
$$
and hence 
$$
\langle  (I-P_{\mc{T}_h,\ell}^{p,m})v,A_\ell w\rangle_{L^2}=\langle  v, A_\ell (I-P_{\mc{T}_h,\ell}^{p,m})w\rangle_{L^2}.
$$
In particular, 
$$
[A_\ell (I-P_{\mc{T}_h,\ell}^{p,m})]^*=A_\ell(I-P_{\mc{T}_h,\ell}^{p,m})\qquad \Rightarrow\qquad  (I-P_{\mc{T}_h,\ell}^{p,m})=A_\ell^{-1}(I-P_{\mc{T}_h,\ell}^{p,m})^*A_\ell.
$$
So that 
\begin{align*}
\|(I-P_{\mc{T}_h,\ell}^{p,m})\|_{H_k^\ell \to H_k^{-s}}&\leq \|A_\ell^{-1}\|_{H_k^{-2\ell-s}\to H_k^{-s}}\|(I-P_{\mc{T}_h,\ell}^{p,m})^*\|_{H_k^{-\ell}\to H_k^{-2\ell-s}}\|A_\ell\|_{H_k^\ell \to H_k^{-\ell}}\\
&\leq C \|A_\ell^{-1}\|_{H_k^{-2\ell-s}\to H_k^{-s}}\|(I-P_{\mc{T}_h,\ell}^{p,m})\|_{H_k^{2\ell+s}\to H_k^{\ell}}\\
&=C  \|A_\ell^{-1}\|_{H_k^{s}\to H_k^{s+2\ell}}\|(I-P_{\mc{T}_h,\ell}^{p,m})\|_{H_k^{2\ell+s}\to H_k^{\ell}}\\
&\leq C\|(I-P_{\mc{T}_h,\ell}^{p,m})\|_{H_k^{2\ell+s}\to H_k^{\ell}},
\end{align*}
where the last line follows since $s+2\ell\leq \max(2(p+1),2\ell+s)$.

\end{proof}

Next, we prove that Theorem~\ref{t:MHs} is, in fact, optimal. 
\begin{proof}[Proof of optimality of Theorem~\ref{t:MHs}]
Observe that by Lemma~\ref{l:duality} and~\eqref{e:conforming}
\begin{align*}
\|I-P_{\mc{T}_h,\ell}^{p,m}\|_{H_k^{p+1}\to H_k^{-s}}&=\|(I-P_{\mc{T}_h,\ell}^{p,m})(I-P_{\mc{T}_h,\ell}^{p,m})u\|_{H_k^{p+1}\to H_k^{-s}}\\
&\leq \|(I-P_{\mc{T}_h,\ell}^{p,m})\|_{H_k^\ell\to H_k^{-s}}\|(I-P_{\mc{T}_h,\ell}^{p,m})\|_{H_k^{p+1}\to H_k^\ell}\\
&\leq C\|(I-P_{\mc{T}_h,\ell}^{p,m})\|_{H_k^{p+1}\to H_k^\ell}^2\leq C(hk)^{2(p+1-\ell)}.
\end{align*}
The optimality of Theorem~\ref{t:MHs} then follows from the definition of oscillating between $\Xi_Lk$ and $\Xi_Hk$. 
\end{proof}

\bibliography{biblio}
 \bibliographystyle{alpha}
\end{document}